\documentclass[12pt]{article}%
\usepackage{amsmath}
\usepackage{amsfonts}
\usepackage{amssymb}
\usepackage{graphicx}
\usepackage{textcomp}
\usepackage{hyperref}%
\providecommand{\U}[1]{\protect\rule{.1in}{.1in}}
\let\Horig\H
\usepackage{cite}

\newtheorem{theorem}{Theorem}

\newtheorem{corollary}[theorem]{Corollary}

\newtheorem{definition}[theorem]{Definition}

\newtheorem{lemma}[theorem]{Lemma}

\newtheorem{proposition}[theorem]{Proposition}
\newtheorem{remark}[theorem]{Remark}

\newenvironment{proof}[1][Proof]{\noindent\textbf{#1.} }{\ \rule{0.5em}{0.5em}}
\begin{document}

\title{Asymptotic behavior of extremals for fractional Sobolev inequalities
associated with singular problems}
\author{G. Ercole\thanks{Corresponding author} , G.A. Pereira, and R. Sanchis\\{\small grey@mat.ufmg.br, gilbertoapereira@yahoo.com.br,
rsanchis@mat.ufmg.br\bigskip}\\Universidade Federal de Minas Gerais, Belo Horizonte, MG,\\30.123-970, Brazil. }
\maketitle

\begin{abstract}
Let $\Omega$ be a smooth, bounded domain of $\mathbb{R}^{N}$, $\omega$ be a
positive, $L^{1}$-normalized function, and $0<s<1<p.$ We study the asymptotic
behavior, as $p\rightarrow\infty,$ of the pair $\left(  \sqrt[p]{\Lambda_{p}%
},u_{p}\right)  ,$ where $\Lambda_{p}$ is the best constant $C$ in the Sobolev
type inequality
\[
C\exp\left(  \int_{\Omega}(\log\left\vert u\right\vert ^{p})\omega
\mathrm{d}x\right)  \leq\left[  u\right]  _{s,p}^{p}\quad\forall\,u\in
W_{0}^{s,p}(\Omega)
\]
and $u_{p}$ is the positive, suitably normalized extremal function
corresponding to $\Lambda_{p}$. We show that the limit pairs are closely
related to the problem of minimizing the quotient $\left\vert u\right\vert
_{s}/\exp\left(  \int_{\Omega}(\log\left\vert u\right\vert )\omega
\mathrm{d}x\right)  ,$ where $\left\vert u\right\vert _{s}$ denotes the
$s$-H\"{o}lder seminorm of a function $u\in C_{0}^{0,s}(\overline{\Omega}).$

\end{abstract}

\noindent\textbf{2010 AMS Classification.} 35D40, 35R11, 35J60.


\section{Introduction}

Let $\Omega$ be a smooth (at least Lipschitz) domain of $\mathbb{R}^{N}$ and
consider the fractional Sobolev space
\[
W_{0}^{s,p}(\Omega):=\left\{  u\in L^{p}(\mathbb{R}^{N}):u=0\ \mathrm{in}%
\ \mathbb{R}^{N}\setminus\Omega\quad\mathrm{and}\quad\left[  u\right]
_{s,p}<\infty\right\}  ,\quad0<s<1<p,
\]
where
\[
\left[  u\right]  _{s,p}:=\left(  \int_{\mathbb{R}^{N}}\int_{\mathbb{R}^{N}%
}\frac{\left\vert u(x)-u(y)\right\vert ^{p}}{\left\vert x-y\right\vert
^{N+sp}}\mathrm{d}x\mathrm{d}y\right)  ^{\frac{1}{p}}.
\]

It is well-known that the Gagliardo seminorm $\left[  \cdot\right]  _{s,p}$ is
a norm in $W_{0}^{s,p}(\Omega)$ and that this Banach space is uniformly
convex. Actually,
\[
W_{0}^{s,p}(\Omega)=\overline{C_{c}^{\infty}(\Omega)}^{\left[  \cdot\right]
_{s,p}}.
\]

Let $\omega$ be a nonnegative function in $L^{1}(\Omega)$ satisfying
$\left\Vert \omega\right\Vert _{L^{1}(\Omega)}=1$ and define
\[
\mathcal{M}_{p}:=\left\{  u\in W_{0}^{s,p}(\Omega):\int_{\Omega}%
(\log\left\vert u\right\vert )\omega\mathrm{d}x=0\right\}
\]
and%
\begin{equation}
\Lambda_{p}:=\inf\left\{  \left[  u\right]  _{s,p}^{p}:u\in\mathcal{M}%
_{p}\right\}  . \label{Lambdap}%
\end{equation}

In the recent paper \cite{EPmana18} is proved that $\Lambda_{p}>0$ and that
\begin{equation}
\Lambda_{p}\exp\left(  \int_{\Omega}(\log\left\vert u\right\vert ^{p}%
)\omega\mathrm{d}x\right)  \leq\left[  u\right]  _{s,p}^{p}\quad\forall\,u\in
W_{0}^{s,p}(\Omega), \label{sineq}%
\end{equation}
provided that $\Lambda_{p}<\infty.$ Moreover, the equality in this Sobolev
type inequality holds if, and only if, $u$ is a scalar multiple of the
function $u_{p}\in\mathcal{M}_{p}$ which is the only weak solution of the
problem
\begin{equation}
\left\{
\begin{array}
[c]{lll}%
\left(  -\Delta_{p}\right)  ^{s}u=\Lambda_{p}u^{-1}\omega & \mathrm{in} &
\Omega\\
u>0 & \mathrm{in} & \Omega\\
u=0 & \mathrm{in} & \mathbb{R}^{N}\setminus\Omega.
\end{array}
\right.  \label{extremal}%
\end{equation}
Here, $\left(  -\Delta_{p}\right)  ^{s}$ is the $s$-fractional $p$-Laplacian,
formally defined by%
\[
\left(  -\Delta_{p}\right)  ^{s}u(x)=-2\int_{\mathbb{R}^{N}}\dfrac{\left\vert
u(y)-u(x)\right\vert ^{p-2}(u(y)-u(x))}{\left\vert y-x\right\vert ^{N+sp}%
}\mathrm{d}y.
\]

We recall that a weak solution of the equation in (\ref{extremal}) is a
function $u\in W_{0}^{s,p}(\Omega)$ satisfying
\[
\left\langle \left(  -\Delta_{p}\right)  ^{s}u,\varphi\right\rangle
=\Lambda_{p}\int_{\Omega}u^{-1}\varphi\omega\mathrm{d}x\quad\forall
\,\varphi\in W_{0}^{s,p}(\Omega),
\]
where
\[
\left\langle \left(  -\Delta_{p}\right)  ^{s}u,\varphi\right\rangle
:=\int_{\mathbb{R}^{N}}\int_{\mathbb{R}^{N}}\dfrac{\left\vert
u(x)-u(y)\right\vert ^{p-2}(u(x)-u(y))(\varphi(x)-\varphi(y))}{\left\vert
x-y\right\vert ^{N+sp}}\mathrm{d}x\mathrm{d}y
\]
is the expression of $\left(  -\Delta_{p}\right)  ^{s}$ as an operator from
$W_{0}^{s,p}(\Omega)$ into its dual.

The purpose of this paper is to determine both the asymptotic behavior of the
pair $\left(  \sqrt[p]{\Lambda_{p}},u_{p}\right)  $, as $p\rightarrow\infty$,
and the corresponding limit problem of (\ref{extremal}). In our study
$s\in(0,1)$ is kept fixed.

After introducing, in Section \ref{Sec0}, the notation used throughout the
paper, we prove in Section \ref{sec1} that $\Lambda_{p}<\infty$ by
constructing a function $\xi\in C_{0}^{0,1}(\overline{\Omega})\cap
\mathcal{M}_{p}.$ In the simplest case $\omega\equiv\left\vert \Omega
\right\vert ^{-1}$ this was made in \cite{EPjam19} where the inequality
(\ref{sineq}) corresponding to the standard Sobolev Space $W_{0}^{1,p}%
(\Omega)$ has been derived.

In Section \ref{sec2}, we show that the limit problem is closely related to
the problem of minimizing the quotient
\[
Q_{s}(u):=\frac{\left\vert u\right\vert _{s}}{\exp\left(  \int_{\Omega}%
(\log\left\vert u\right\vert )\omega\mathrm{d}x\right)  }%
\]
on the Banach space $\left(  C_{0}^{0,s}(\overline{\Omega}),\left\vert
\cdot\right\vert _{s}\right)  $ of the $s$-H\"{o}lder continuous functions in
$\overline{\Omega}$ that are zero on the boundary $\partial\Omega.$ Here,
$\left\vert u\right\vert _{s}$ denotes the $s$-H\"{o}lder seminorm of $u$ (see
(\ref{bholder})).

We prove that if $p_{n}\rightarrow\infty$ then (up to a subsequence)
\[
u_{p_{n}}\rightarrow u_{\infty}\in C_{0}^{0,s}(\overline{\Omega}%
)\ \mathrm{uniformly\ in}\ \overline{\Omega},\quad\mathrm{and}\quad
\sqrt[p_{n}]{\Lambda_{p_{n}}}\rightarrow\left\vert u_{\infty}\right\vert
_{s}.
\]
Moreover, the limit function $u_{\infty}$ satisfies%
\[
\int_{\Omega}(\log\left\vert u_{\infty}\right\vert )\omega\mathrm{d}%
x\geq0\quad\mathrm{and}\quad Q_{s}(u_{\infty})\leq Q_{s}(u)\quad\forall\ u\in
C_{0}^{0,s}(\overline{\Omega})\setminus\left\{  0\right\}
\]
and the only minimizers of the quotient $Q_{s}$ are the scalar multiples of
$u_{\infty}.$

One of the difficulties we face in Section \ref{sec2} is that $C_{c}^{\infty
}(\Omega)$ is not dense in $\left(  C_{0}^{0,s}(\Omega),\left\vert
\cdot\right\vert _{s}\right)  .$ This makes it impossible to directly exploit
the fact that $u_{p}$ is a weak solution of (\ref{extremal}). We overcome this
issue by using a convenient technical result proved in \cite[Lemma 3.2]{LiChu}
and employed in \cite{BLP} to deal with a similar approximation matter.

In Section \ref{sec3}, motived by \cite{Chamb, FPL, LindLind}, we derive the
limit problem of (\ref{extremal}). Assuming that $\omega$ is continuous and
positive in $\Omega$ we prove that $u_{\infty}$ is a viscosity solution of%
\[
\left\{
\begin{array}
[c]{lll}%
\mathcal{L}_{\infty}^{-}u+\left\vert u\right\vert _{s}=0 & \mathrm{in} &
\Omega\\
u=0 & \mathrm{in} & \mathbb{R}^{N}\setminus\Omega
\end{array}
\right.
\]
where%
\[
\left(  \mathcal{L}_{\infty}^{-}u\right)  (x):=\inf_{y\in\mathbb{R}%
^{N}\setminus\left\{  x\right\}  }\frac{u(y)-u(x)}{\left\vert y-x\right\vert
^{s}}.
\]

We also show $u_{\infty}$ is a viscosity supersolution of
\[
\left\{
\begin{array}
[c]{lll}%
\mathcal{L}_{\infty}u=0 & \mathrm{in} & \Omega\\
u=0 & \mathrm{in} & \mathbb{R}^{N}\setminus\Omega
\end{array}
\right.
\]
where
\[
\mathcal{L}_{\infty}:=\mathcal{L}_{\infty}^{+}+\mathcal{L}_{\infty}^{-}%
\]
and
\[
\left(  \mathcal{L}_{\infty}^{+}u\right)  (x):=\sup_{y\in\mathbb{R}%
^{N}\setminus\left\{  x\right\}  }\frac{u(y)-u(x)}{\left\vert y-x\right\vert
^{s}}.
\]
This fact guarantees that $u_{\infty}>0$ in $\Omega.$

The existing literature on the asymptotic behavior (as $p\rightarrow\infty$)
of solutions of problems involving the $p$-Laplacian is most focused on the
local version of the operator, that is, on the problem%
\begin{equation}
\left\{
\begin{array}
[c]{lll}%
-\Delta_{p}u=f(x,u) & \mathrm{in} & \Omega\\
u=0 & \mathrm{on} & \mathbb{\partial}\Omega
\end{array}
\right.  \label{local}%
\end{equation}
where $\Delta_{p}u=\operatorname{div}\left(  \left\vert \nabla u\right\vert
^{p-2}\nabla u\right)  $ is the standard $p$-Laplacian. This kind of
asymptotic behavior has been studied for at least three decades (see
\cite{BDM, Fuka, JLM}) and many new results, adding the dependence of $p$ in
the term $f(x,u),$ are still being produced (see \cite{ChaPe, Chapa, ChaPa2,
EPmana16}). The solutions of (\ref{local}) are obtained in the natural Sobolev
space $W_{0}^{1,p}(\Omega)$ and an important property related to this space,
crucial in the study of the asymptotic behavior of the corresponding family of
solutions $\left\{  u_{p}\right\}  ,$ is the inclusion
\[
W_{0}^{1,p_{2}}(\Omega)\subset W_{0}^{1,p_{1}}(\Omega)\quad\mathrm{whenever}%
\quad1<p_{1}<p_{2}.
\]
It allows us to show that any uniform limit function $u_{\infty}$ of the
sequence $\left\{  u_{p_{n}}\right\}  $ (with $p_{n}\rightarrow\infty$) is
admissible as a test function in the weak formulation of (\ref{local}), so
that $u_{\infty}$ inherits certain properties of the functions of $\left\{
u_{p_{n}}\right\}  .$

Since the inclusion $W_{0}^{s,p_{2}}(\Omega)\subset W_{0}^{s,p_{1}}(\Omega)$
does not hold when $0<s<1<p_{1}<p_{2}$ (see \cite{Miro}) the asymptotic
behavior, as $p\rightarrow\infty,$ of the solutions of the problem
\begin{equation}
\left\{
\begin{array}
[c]{lll}%
(-\Delta_{p})^{s}u=f(x,u) & \mathrm{in} & \Omega\\
u=0 & \mathrm{in} & \mathbb{R}^{N}\setminus\Omega
\end{array}
\right.  \label{nonlocal}%
\end{equation}
is more difficult to be determined. For example, in the case considered in the
present paper ($f(x,u)=\omega(x)/u$) we cannot ensure that the property
\[
\int_{\Omega}(\log\left\vert u_{p_{n}}\right\vert )\omega\mathrm{d}x=0
\]
is inherited by the limit function $u_{\infty}$ (see Remark \ref{kleq0}).
Actually, we are able to prove only that
\[
\int_{\Omega}(\log u_{\infty})\omega\mathrm{d}x\geq0.
\]
As a consequence, the limit functions of the family $\left\{  u_{p}\right\}
_{p>1}$ might not be unique.

The study of the asymptotic behavior, as $p\rightarrow\infty,$ of the
solutions of (\ref{nonlocal}) is quite recent and restricted to few works. In
\cite{LindLind} the authors considered $f(x,u)=\lambda_{p}\left\vert
u\right\vert ^{p-2}u$ where $\lambda_{p}$ is the first eigenvalue of the
$s$-fractional $p$-Laplacian. Among other results, they proved that%
\[
\lim_{p\rightarrow\infty}\sqrt[p]{\lambda_{p}}=R^{-s},
\]
where $R$ is the radius of the largest ball inscribed in $\Omega,$ and that
limit function $u_{\infty}$ of the family $\left\{  u_{p}\right\}  $ is a
positive viscosity solution of
\[
\max\left\{  \mathcal{L}_{\infty}u\ ,\ \mathcal{L}_{\infty}^{-}u+R^{-s}%
u\right\}  =0.
\]

The equation in (\ref{nonlocal}) with $f=0$ and under the nonhomogeneous
boundary condition $u=g$ in $\mathbb{R}^{N}\setminus\Omega$ was first studied
in \cite{Chamb}. It is shown that the limit function is an optimal
$s$-H\"{o}lder extension of $g\in C^{0,s}(\partial\Omega)$ and also a
viscosity solution of the equation%
\[
\mathcal{L}_{\infty}u=0\quad\mathrm{in}\ \partial\Omega.
\]
Moreover, some tools for studying the behavior as $p\rightarrow\infty$ of the
solutions of (\ref{nonlocal}) are developed there.

In \cite{FPL}, also under the boundary condition $u=g$ in $\mathbb{R}%
^{N}\setminus\Omega,$ the cases $f=f(x)$ and $f=f(u)=\left\vert u\right\vert
^{\theta(p)-2}u$ with $\Theta:=\lim_{p\rightarrow\infty}\theta(p)/p<1$ are
studied. In the first case, different limit equations involving the operators
$\mathcal{L}_{\infty},$ $\mathcal{L}_{\infty}^{+}$ and $\mathcal{L}_{\infty
}^{-}$ are derived according to the sign of the function $f(x),$ what
resembles the known results obtained in \cite{BDM}, where the standard
$p$-Laplacian is considered. For example, the limit function $u_{\infty}$ is a
viscosity solution of
\[
-\mathcal{L}_{\infty}^{-}u=1\quad\mathrm{in}\ \left\{  f>0\right\}  .
\]
As for the second case, the limit equation is%
\[
\min\left\{  -\mathcal{L}_{\infty}^{-}u-u^{\Theta},-\mathcal{L}_{\infty
}u\right\}  =0
\]
which is consistent with the limit equation obtained in \cite{ChaPe} for the
standard $p$-Laplacian and $f(u)=\left\vert u\right\vert ^{\theta(p)-2}u$
satisfying $\Theta:=\lim_{p\rightarrow\infty}\theta(p)/p<1.$

\section{Notation\label{Sec0}}

The ball centered at $x\in\mathbb{R}^{N}$ with radius $\rho$ is denoted by
$B(x,\rho)$ and $\delta$ stands for the distance function to the boundary
$\partial\Omega,$ defined by%
\[
\delta(x):=\min_{y\in\partial\Omega}\left\vert x-y\right\vert ,\quad
x\in\overline{\Omega}.
\]

We recall that $\delta\in C_{0}^{0,1}(\overline{\Omega})$ and satisfies
$\left\vert \nabla\delta\right\vert =1$ a.e. in $\Omega.$ Here,
\[
C_{0}^{0,\beta}(\overline{\Omega}):=\left\{  u\in C^{0,\beta}(\overline
{\Omega}):\,u=0\,\mathrm{on}\,\partial\Omega\right\}  ,\quad0<\beta\leq1,
\]
where $C^{0,\beta}(\overline{\Omega})$ is the well-known $\beta$-H\"{o}lder
space endowed with the norm
\[
\left\Vert u\right\Vert _{0,\beta}=\left\Vert u\right\Vert _{\infty
}+\left\vert u\right\vert _{\beta}%
\]
with $\left\Vert u\right\Vert _{\infty}$ denoting the sup norm of $u$ and
$\left\vert u\right\vert _{\beta}$ denoting the $\beta$-H\"{o}lder seminorm,
that is,
\begin{equation}
\left\vert u\right\vert _{\beta}:=\sup_{x,y\in\overline{\Omega},x\not =y}%
\frac{\left\vert u(x)-u(y)\right\vert }{\left\vert x-y\right\vert ^{\beta}}.
\label{bholder}%
\end{equation}

We recall that $\left(  C_{0}^{0,\beta}(\overline{\Omega}),\left\vert
\cdot\right\vert _{\beta}\right)  $ is a Banach space. The fact that the
$\beta$-H\"{o}lder seminorm $\left\vert \cdot\right\vert _{\beta}$ is a norm
in $C_{0}^{0,\beta}(\overline{\Omega})$ equivalent to $\left\Vert u\right\Vert
_{0,\beta}$ is a consequence of the estimate
\[
\left\Vert u\right\Vert _{\infty}\leq\left\vert u\right\vert _{\beta
}\left\Vert \delta\right\Vert _{\infty}^{\beta}\quad\forall\,u\in
C_{0}^{0,\beta}(\overline{\Omega}),
\]
which in turn follows from the following
\begin{equation}
\left\vert u(x)\right\vert =\left\vert u(x)-u(y_{x})\right\vert \leq\left\vert
u\right\vert _{\beta}\left\vert x-y_{x}\right\vert ^{\beta}=\left\vert
u\right\vert _{\beta}\delta(x)^{\beta}\quad\forall\,x\in\Omega, \label{udelta}%
\end{equation}
where $y_{x}\in\partial\Omega$ is such that $\delta(x)=\left\vert
x-y_{x}\right\vert .$

We also define
\[
C_{c}^{\infty}(\Omega):=\left\{  u\in C^{\infty}(\Omega):\operatorname*{supp}%
(f)\subset\subset\Omega\right\}
\]
where
\[
\operatorname*{supp}(u):=\left\{  x\in\Omega:u(x)\not =0\right\}
\]
is the support of $u$ and $X\subset\subset Y$ means that $\overline{X}$ is a
compact subset of $Y$. Analogously, we define $E_{c}$ if $E$ is a space of
functions (e.g. $C_{c}(\mathbb{R}^{N}),$ $C_{c}(\mathbb{R}^{N};\mathbb{R}%
^{N}),$ $C_{c}^{0,\beta}(\overline{\Omega})$).

\section{Finiteness of $\Lambda_{p}$\label{sec1}}

Let us recall the Federer's co-area formula (see \cite{Fed})%
\[
\int_{\Omega}g(x)\left\vert \nabla f(x)\right\vert \mathrm{d}x=\int_{-\infty
}^{\infty}\left(  \int_{f^{-1}\left\{  t\right\}  }g(x)\mathrm{d}%
\mathcal{H}_{N-1}\right)  \mathrm{d}t,
\]
which holds whenever $g\in L^{1}(\Omega)$ and $f\in C^{0,1}(\overline{\Omega
})$. (In this formula $\mathcal{H}_{N-1}$ stands for the $(N-1)$-dimensional
Hausdorff measure).

In the particular case $f=\delta$ the above formula becomes%
\begin{equation}
\int_{\Omega}g(x)\mathrm{d}x=\int_{0}^{\left\Vert \delta\right\Vert _{\infty}%
}\left(  \int_{\delta^{-1}\left\{  t\right\}  }g(x)\mathrm{d}\mathcal{H}%
_{N-1}\right)  \mathrm{d}t.\label{coa}%
\end{equation}

\begin{proposition}
\label{mfin}Let $\omega\in L^{1}(\Omega)$ such that
\begin{equation}
\int_{\Omega}\omega\mathrm{d}x=1\quad\mathrm{and}\quad\omega\geq
0\quad\mathrm{a.e.}\,\mathrm{in}\,\Omega.\label{weight}%
\end{equation}
There exists a nonnegative function $\xi\in C(\overline{\Omega})$ that
vanishes on the boundary $\partial\Omega$ and satisfies%
\[
\int_{\Omega}(\log\left\vert \xi\right\vert )\omega\mathrm{d}x=0.
\]
If, in addition,
\begin{equation}
K_{\epsilon}:=\operatorname*{ess}_{0\leq t\leq\epsilon}\int_{\delta
^{-1}\left\{  t\right\}  }\omega\mathrm{d}\mathcal{H}_{N-1}<\infty\label{Keps}%
\end{equation}
for some $\epsilon>0,$ then $\xi\in C_{0}^{0,1}(\overline{\Omega}).$
\end{proposition}

\begin{proof}
Let $\sigma:[0,\left\Vert \delta\right\Vert _{\infty}]\rightarrow\lbrack0,1]$
be the $\omega$-distribution associated with $\delta,$ that is,
\[
\sigma(t):=\int_{\Omega_{t}}\omega\mathrm{d}x,\quad t\in\lbrack0,\left\Vert
\delta\right\Vert _{\infty}]
\]
where%
\[
\Omega_{t}:=\left\{  x\in\Omega:\delta(x)>t\right\}
\]
is the $t$-superlevel set of $\delta.$

We remark that $\sigma$ is continuous at each point $t\in\lbrack0,\left\Vert
\delta\right\Vert _{\infty}]$ since the $t$-level set $\delta^{-1}\left\{
t\right\}  $ has Lebesgue measure zero. This follows, for example, from the
Lebesgue density theorem (see \cite{Erdos}, where the distance function to a
general closed set in $\mathbb{R}^{N}$ is considered).

Thus, there exists a nonincreasing sequence $\left\{  t_{n}\right\}
\subset\lbrack0,\left\Vert \delta\right\Vert _{\infty}]$ such that%
\[
\sigma(t_{n})=1-\frac{1}{2^{n}}.
\]

Now, choose a nondecreasing, piecewise linear function $\varphi\in
C([0,\left\Vert \delta\right\Vert _{\infty}])$ satisfying
\[
\varphi(0)=0\quad\mathrm{and}\quad\varphi(t_{n})=\frac{1}{2^{n}},
\]
and take the function%
\[
\xi_{1}:=\varphi\circ\delta\in C_{0}(\overline{\Omega}).
\]

Taking into account that
\[
t_{n+1}\leq\delta(x)\leq t_{n}\quad\mathrm{a.e.}\,x\in\Omega_{t_{n+1}%
}\setminus\Omega_{t_{n}}%
\]
one has
\[
\frac{1}{2^{n+1}}=\varphi(t_{n+1})\leq\xi_{1}(x)\leq\varphi(t_{n})=\frac
{1}{2^{n}}\quad\mathrm{a.e.}\,x\in\Omega_{t_{n+1}}\setminus\Omega_{t_{n}}.
\]
Consequently,
\begin{align*}
\int_{\Omega}\left\vert \xi_{1}\right\vert ^{\epsilon}\omega\mathrm{d}x  &
\geq\int_{\Omega_{t_{1}}}\left\vert \xi_{1}\right\vert ^{\epsilon}%
\omega\mathrm{d}x+\sum_{k=1}^{n}\int_{\Omega_{t_{k+1}}\setminus\Omega_{t_{k}}%
}\left\vert \xi_{1}\right\vert ^{\epsilon}\omega\mathrm{d}x\\
&  \geq\frac{1}{2^{\epsilon}}\int_{\Omega_{t_{1}}}\omega\mathrm{d}x+\sum
_{k=1}^{n}\frac{1}{2^{\epsilon(k+1)}}\int_{\Omega_{t_{k+1}}\setminus
\Omega_{t_{k}}}\omega\mathrm{d}x\\
&  =\frac{1}{2^{\epsilon}}\sigma(t_{1})+\sum_{k=1}^{n}\frac{1}{2^{\epsilon
(k+1)}}\left(  \sigma(t_{k+1})-\sigma(t_{k})\right) \\
&  =\frac{1}{2^{\epsilon}}\frac{1}{2}+\sum_{k=1}^{n}\frac{1}{2^{\epsilon
(k+1)}}\frac{1}{2^{k+1}}=\sum_{k=1}^{n+1}\left(  (1/2)^{\epsilon+1}\right)
^{k}.
\end{align*}

It follows that%
\[
\lim_{\epsilon\rightarrow0}\left(  \int_{\Omega}\left\vert \xi_{1}\right\vert
^{\epsilon}\omega\mathrm{d}x\right)  ^{\frac{1}{\epsilon}}\geq\lim
_{\epsilon\rightarrow0}\left(  \sum_{k=1}^{\infty}\left(  (1/2)^{\epsilon
+1}\right)  ^{k}\right)  ^{\frac{1}{\epsilon}}=\lim_{\epsilon\rightarrow
0}\left(  \frac{(1/2)^{\epsilon+1}}{1-(1/2)^{\epsilon+1}}\right)  ^{\frac
{1}{\epsilon}}=\frac{1}{4}.
\]

Taking $\xi:=k\xi_{1}$ with
\[
k=\lim_{\epsilon\rightarrow0}\left(  \int_{\Omega}\left\vert \xi
_{1}\right\vert ^{\epsilon}\omega\mathrm{d}x\right)  ^{-\frac{1}{\epsilon}}%
\]
we obtain, by L'H\^{o}pital's rule,
\[
1=\lim_{\epsilon\rightarrow0^{+}}\left(  \int_{\Omega}\left\vert
\xi\right\vert ^{\epsilon}\omega\mathrm{d}x\right)  ^{\frac{1}{\epsilon}}%
=\exp\left(  \int_{\Omega}(\log\left\vert \xi\right\vert )\omega
\mathrm{d}x\right)  .
\]
Hence,%
\[
\int_{\Omega}(\log\left\vert \xi\right\vert )\omega\mathrm{d}x=0.
\]

We now prove that $\xi_{1}\in C^{0,1}(\overline{\Omega})$ under the additional
hypothesis (\ref{Keps}). Since the nondecreasing function $\varphi$ can be
chosen such that $\varphi^{\prime}$ is bounded in any closed interval
contained in $(0,\left\Vert \delta\right\Vert _{\infty}],$ we can assume that
$\nabla\xi_{1}\in L_{\operatorname{loc}}^{\infty}(\Omega)$ (note that
$\left\vert \nabla\xi_{1}\right\vert =\left\vert \varphi^{\prime}%
(\delta)\nabla\delta\right\vert =\left\vert \varphi^{\prime}(\delta
)\right\vert $ a.e. in $\Omega$).

Thus, it suffices to show that the quotient
\[
Q(x,y):=\frac{\left\vert \xi_{1}(x)-\xi_{1}(y)\right\vert }{\left\vert
x-y\right\vert }%
\]
is bounded uniformly with respect to $y\in\partial\Omega$ and $x\in
\Omega_{\epsilon}^{c}:=\left\{  x\in\overline{\Omega}:\delta(x)\leq
\epsilon\right\}  ,$ where $\epsilon$ is given by (\ref{Keps}).

Let $x\in\Omega_{\epsilon}^{c}$ and $y\in\partial\Omega$ be fixed and chose
$n\in\mathbb{N}$ sufficiently large such that
\[
t_{n+1}<\delta(x)\leq t_{n}\leq\epsilon.
\]
Since $\xi_{1}(y)=0$ and $\varphi$ is nondecreasing one has%
\[
\left\vert \xi_{1}(x)-\xi_{1}(y)\right\vert =\xi_{1}(x)\leq\varphi
(t_{n})=\frac{1}{2^{n}}.
\]
Moreover,
\[
t_{n+1}<\delta(x)\leq\left\vert x-y\right\vert .
\]

Hence,
\[
Q(x,y)\leq\frac{1}{2^{n}t_{n+1}}\quad\mathrm{whenever}\,y\in\partial
\Omega\,\mathrm{and}\,x\in\Omega_{\epsilon}^{c}.
\]

Applying the co-area formula (\ref{coa}) with $g=\omega$ and $\Omega
=\Omega_{t_{n}+1}^{c}$ we find
\[
\frac{1}{2^{n+1}}=\int_{\Omega_{t_{n}+1}^{c}}\omega\mathrm{d}x=\int
_{0}^{t_{n+1}}\left(  \int_{\delta^{-1}\left\{  t\right\}  }\omega
\mathrm{d}\mathcal{H}_{N-1}\right)  \mathrm{d}t\leq K_{\epsilon}t_{n+1}.
\]

It follows that%
\begin{equation}
Q(x,y)\leq\frac{1}{2^{n}t_{n+1}}\leq\frac{K_{\epsilon}2^{n+1}}{2^{n}%
}=2K_{\epsilon}\quad\mathrm{whenever}\,y\in\partial\Omega\,\mathrm{and}%
\,x\in\Omega_{\epsilon}^{c}, \label{weyl}%
\end{equation}
concluding thus the proof that $\xi_{1}\in C^{0,1}(\overline{\Omega}).$
\end{proof}

\begin{remark}
The estimate (\ref{weyl}) can also be obtained from the Weyl's Formula (see
\cite{Gray}) provided that $\omega$ is bounded on an $\epsilon$-tubular
neighborhood of $\partial\Omega.$
\end{remark}

In the remaining of this section $\xi$ denotes the function obtained in
Proposition \ref{mfin} extended as zero outside $\Omega.$ So,
\[
\xi\in C_{0}^{0,1}(\overline{\Omega})\quad\mathrm{and}\quad\int_{\Omega}%
(\log\left\vert \xi\right\vert )\omega\mathrm{d}x=0.
\]
Since $C_{0}^{0,1}(\overline{\Omega})\subseteq W_{0}^{1,p}(\Omega)\subseteq
W_{0}^{s,p}(\Omega)$ we have $\xi\in\mathcal{M}_{p}$ (for a proof of the
second inclusion see \cite{Guide}). Therefore,
\begin{equation}
\Lambda_{p}\leq\left[  \xi\right]  _{s,p}^{p}\quad\forall\,p>1. \label{Lamfin}%
\end{equation}

Combining (\ref{Lamfin}) with the results proved in \cite[Section 4]{EPmana18}
(which requires $\omega\in L^{r}(\Omega),$ for some $r>1$) we have the
following theorem.

\begin{theorem}
\label{teomana}Let $\omega$ be a function in $L^{r}(\Omega),$ for some $r>1,$
satisfying (\ref{weight})-(\ref{Keps}). For each $p>1,$ the infimum
$\Lambda_{p}$ in (\ref{Lambdap}) is attained by a function $u_{p}%
\in\mathcal{M}_{p}$ which is the only positive weak solution of
\[
\left(  -\Delta_{p}\right)  ^{s}u=\Lambda_{p}u^{-1}\omega,\quad u\in
W_{0}^{s,p}(\Omega).
\]

\end{theorem}

Summarizing,
\begin{equation}
\left[  u_{p}\right]  _{s,p}^{p}=\Lambda_{p}:=\min\left\{  \left[  u\right]
_{s,p}^{p}:u\in\mathcal{M}_{p}\right\}  \leq\left[  \xi\right]  _{s,p}%
^{p}\quad\forall\,p>1, \label{up1}%
\end{equation}
and $u_{p}$ is the unique function in $W_{0}^{1,p}(\Omega)$ satisfying
\[
u_{p}>0\quad\mathrm{in}\,\Omega\quad\mathrm{and}\quad\left\langle \left(
-\Delta_{p}\right)  ^{s}u_{p},\phi\right\rangle =\Lambda_{p}\int_{\Omega
}\omega(u_{p})^{-1}\phi\mathrm{d}x\quad\forall\,\phi\in W_{0}^{s,p}(\Omega).
\]

We also have%
\[
0<\sqrt[p]{\Lambda_{p}}\leq\frac{\left[  u\right]  _{s,p}}{\exp\left(
\int_{\Omega}(\log\left\vert u\right\vert )\omega\mathrm{d}x\right)  }%
\quad\forall\,u\in W_{0}^{s,p}(\Omega),
\]
since the quotient is homogeneous.

\begin{remark}
\label{-inf}It is worth pointing out that
\begin{equation}
\int_{\Omega}(\log\left\vert u\right\vert )\omega\mathrm{d}x=-\infty
\label{loginf}%
\end{equation}
for any function $u\in L^{\infty}(\Omega)$ whose $\operatorname*{supp}u$ is a
proper subset of $\operatorname*{supp}\omega.$ Indeed, in this case we have%
\[
0\leq\exp\left(  \int_{\Omega}(\log\left\vert u\right\vert )\omega
\mathrm{d}x\right)  =\lim_{t\rightarrow0^{+}}\left(  \int_{\Omega}\left\vert
u\right\vert ^{t}\omega\mathrm{d}x\right)  ^{\frac{1}{t}}\leq\left\Vert
u\right\Vert _{\infty}\lim_{t\rightarrow0^{+}}\left(  \int
_{\operatorname*{supp}\left\vert u\right\vert }\omega\mathrm{d}x\right)
^{\frac{1}{t}}=0.
\]
Thus, if $\omega>0$ almost everywhere in $\Omega$ then (\ref{loginf}) holds
for every $u\in C_{c}^{\infty}(\Omega)\setminus\left\{  0\right\}  .$
\end{remark}

\section{The asymptotic behavior as $p\rightarrow\infty$\label{sec2}}

In this section we assume that the weight $\omega$ satisfies the hypothesis of
Theorem \ref{teomana}. Our goal is to relate the asymptotic behavior (as
$p\rightarrow\infty$) of the pair $\left(  \sqrt[p]{\Lambda_{p}},u_{p}\right)
$ with the problem of minimizing the homogeneous quotient $Q_{s}:C_{0}%
^{0,s}(\overline{\Omega})\setminus\left\{  0\right\}  \rightarrow(0,\infty)$
defined by%
\[
Q_{s}(u):=\frac{\left\vert u\right\vert _{s}}{k(u)}\quad\mathrm{where}\quad
k(u):=\exp\left(  \int_{\Omega}(\log\left\vert u\right\vert )\omega
\mathrm{d}x\right)  .
\]

Note that $k(u)=0$ if, and only if, $u$ satisfies (\ref{loginf}). In
particular, according to Remark \ref{-inf},
\[
\omega>0\quad\mathrm{a.e.}\,\mathrm{in}\,\Omega\Longrightarrow Q_{s}%
(u)=\infty\quad\forall\,u\in C_{c}^{\infty}(\Omega)\setminus\left\{
0\right\}  .
\]

We also observe that
\begin{equation}
0\leq k(u)\leq\int_{\Omega}\left\vert u\right\vert \omega\mathrm{d}%
x<\infty\quad\forall\,u\in C_{0}^{0,s}(\overline{\Omega})\setminus\left\{
0\right\}  , \label{jensen}%
\end{equation}
where the second inequality is consequence of the Jensen's inequality (since
the logarithm is concave):
\begin{equation}
\int_{\Omega}(\log\left\vert u\right\vert )\omega\mathrm{d}x\leq\log\left(
\int_{\Omega}\left\vert u\right\vert \omega\mathrm{d}x\right)  .
\label{jensen1}%
\end{equation}

Now, let us define
\[
\mu_{s}:=\inf_{u\in C_{0}^{0,s}(\overline{\Omega})\setminus\left\{  0\right\}
}Q_{s}(u).
\]

Thanks to the homogeneity of $Q_{s}$ we have%
\[
\mu_{s}=\inf_{u\in\mathcal{M}_{s}}\left\vert u\right\vert _{s}%
\]
where%
\[
\mathcal{M}_{s}:=\left\{  u\in C_{0}^{0,s}(\overline{\Omega}):k(u)=1\right\}
.
\]

Combining (\ref{jensen}) and (\ref{udelta}) we obtain
\[
1\leq\int_{\Omega}\left\vert u\right\vert \omega\mathrm{d}x\leq\left\vert
u\right\vert _{s}\int_{\Omega}\delta^{s}\omega\mathrm{d}x\quad\forall
\,u\in\mathcal{M}_{s},
\]
what yields the following positive lower bound to $\mu_{s}$
\[
\left(  \int_{\Omega}\delta^{s}\omega\mathrm{d}x\right)  ^{-1}\leq\mu_{s}.
\]

In the sequel we show that $\mu_{s}$ is in fact a minimum, attained at a
unique nonnegative function. Before this, let us make an important remark.

\begin{remark}
If $v$ minimizes $\left\vert \cdot\right\vert _{s}$ in $\mathcal{M}_{s}$ the
same holds for $\left\vert v\right\vert ,$ since the function $w=\left\vert
v\right\vert $ belongs to $\mathcal{M}_{s}$ and satisfies $\left\vert
w\right\vert _{s}\leq\left\vert v\right\vert _{s}.$
\end{remark}

\begin{proposition}
\label{minimizer}There exists a unique nonnegative function $v\in
\mathcal{M}_{s}$ such that
\[
\mu_{s}=\left\vert v\right\vert _{s}.
\]

\end{proposition}

\begin{proof}
Let $\left\{  v_{n}\right\}  _{n\in\mathbb{N}}\subset\mathcal{M}_{s}$ be such
that%
\begin{equation}
\lim_{n\rightarrow\infty}\left\vert v_{n}\right\vert _{s}=\mu_{s}.
\label{mmumin}%
\end{equation}
Since the function $w_{n}=\left\vert v_{n}\right\vert $ belongs to
$\mathcal{M}_{s}$ and satisfies $\left\vert w_{n}\right\vert _{s}%
\leq\left\vert v_{n}\right\vert _{s}$ we can assume that $v_{n}\geq0$ in
$\Omega.$

It follows from (\ref{mmumin}) that $\left\{  v_{n}\right\}  _{n\in\mathbb{N}%
}$ is bounded in $C_{0}^{0,s}(\overline{\Omega}).$ Hence, the compactness of
the embedding $C_{0}^{0,s}(\overline{\Omega})\hookrightarrow C_{0}%
(\overline{\Omega})$ allows us to assume (by renaming a subsequence) that
$\left\{  v_{n}\right\}  _{n\in\mathbb{N}}$ converges uniformly to a function
$v\in C_{0}(\overline{\Omega})$. Of course, $v\geq0$ in $\Omega.$

Letting $n\rightarrow\infty$ in the inequality%
\[
\left\vert v_{n}(x)-v_{n}(y)\right\vert \leq\left\vert v_{n}\right\vert
_{s}\left\vert x-y\right\vert ^{s}\quad\forall\,x,y\in\overline{\Omega}%
\]
and taking (\ref{mmumin}) into account we obtain%
\[
\left\vert v(x)-v(y)\right\vert \leq\mu_{s}\left\vert x-y\right\vert ^{s}%
\quad\forall\,x,y\in\overline{\Omega}.
\]
This implies that $v\in C_{0}^{0,s}(\overline{\Omega})$ and
\begin{equation}
\left\vert v\right\vert _{s}\leq\mu_{s}. \label{mumin3}%
\end{equation}

Thus, to prove that $\mu_{s}=\left\vert v\right\vert _{s}$ it suffices to
verify that $v\in\mathcal{M}_{s}.$ Since
\[
1=k(v_{n})=\lim_{\epsilon\rightarrow0^{+}}\left(  \int_{\Omega}\left\vert
v_{n}\right\vert ^{\epsilon}\omega\mathrm{d}x\right)  ^{\frac{1}{\epsilon}%
}\leq\left(  \int_{\Omega}\left\vert v_{n}\right\vert ^{t}\omega
\mathrm{d}x\right)  ^{\frac{1}{t}}\quad\forall\,t>0
\]
the uniform convergence $v_{n}\rightarrow v$ yields%
\[
1\leq\left(  \int_{\Omega}\left\vert v\right\vert ^{t}\omega\mathrm{d}%
x\right)  ^{\frac{1}{t}}\quad\forall\,t>0.
\]
Hence,
\[
1\leq\lim_{t\rightarrow0^{+}}\left(  \int_{\Omega}\left\vert v\right\vert
^{t}\mathrm{d}x\right)  ^{\frac{1}{t}}=k(v).
\]

Thus, noticing that $(k(v))^{-1}v\in\mathcal{M}_{s}$ and taking (\ref{mumin3})
into account we obtain
\[
\mu_{s}\leq\left\vert (k(v))^{-1}v\right\vert _{s}=(k(v))^{-1}\left\vert
v\right\vert _{s}\leq\left\vert v\right\vert _{s}\leq\mu_{s}.
\]
Therefore, $k(v)=1,$ $v\in\mathcal{M}_{s}$ and $\left\vert v\right\vert
_{s}=\mu_{s}.$

Now, let $u\in\mathcal{M}_{s}$ be a nonnegative minimizer of $\left\vert
\cdot\right\vert _{s}$ and consider the convex combination
\[
w:=\theta u+(1-\theta)v\quad\mathrm{with}\quad0<\theta<1.
\]
Since the logarithm is a concave function, we have%
\begin{align*}
\int_{\Omega}(\log w)\omega\mathrm{d}x  &  \geq\int_{\Omega}(\theta
\log(u)+(1-\theta)\log(v))\omega\mathrm{d}x\\
&  =\theta\int_{\Omega}(\log u)\omega\mathrm{d}x+(1-\theta)\int_{\Omega}(\log
v)\omega\mathrm{d}x=0.
\end{align*}
This implies that $c^{-1}w\in\mathcal{M}_{s}$ where $c:=k(w)\geq1.$

Hence,
\[
\mu_{s}\leq c^{-1}\left\vert w\right\vert _{s}\leq\left\vert w\right\vert
_{s}\leq\theta\left\vert u\right\vert _{s}+(1-\theta)\left\vert v\right\vert
_{s}=\theta\mu_{s}+(1-\theta)\mu_{s}=\mu_{s}.
\]
It follows that $c=1$ and the convex combination $w$ minimizes $\left\vert
\cdot\right\vert _{s}$ in $\mathcal{M}_{s}.$ Consequently,
\[
0=\int_{\Omega}\left[  \log(\theta u+(1-\theta)v)\right]  \omega
\mathrm{d}x\geq\int_{\Omega}\left[  \theta\log(u)+(1-\theta)\log(v)\right]
\omega\mathrm{d}x=0.
\]
Since the concavity of the logarithm is strict, one must have $u=Cv$ for some
positive constant $C.$ Taking account that $1=k(u)=Ck(v)=C,$ we have $u=v.$
\end{proof}

From now on, $v_{s}\in\mathcal{M}_{s}$ denotes the only nonnegative minimizer
of $\left\vert \cdot\right\vert _{s}$ on $\mathcal{M}_{s},$ given by
Proposition \ref{minimizer}. The main result of this section, proved in the
sequence, shows that if $p_{n}\rightarrow\infty$ then a subsequence of
$\left\{  u_{p_{n}}\right\}  _{n\in\mathbb{N}}$ converges uniformly to a
scalar multiple of $v_{s,}$ say $u_{\infty}=k_{\infty}v_{s}$ where $k_{\infty
}\geq1.$

In the next section (see (\ref{+-vs})) we show that $u_{\infty}$ is strictly
positive in $\Omega,$ implying thus that $-v_{s}$ and $v_{s}$ are the only
minimizers of $\left\vert \cdot\right\vert _{s}$ on $\mathcal{M}_{s}.$ As
consequence, the minimizers of $Q_{s}$ on $C_{0}^{0,s}(\overline{\Omega
})\setminus\left\{  0\right\}  $ are precisely the scalar multiples of $v_{s}$
(or, equivalently, the scalar multiples of $u_{\infty}$). Further, we derive
an equation satisfied by $v_{s}$ and $\mu_{s}$ in the viscosity sense (see
Corollary \ref{viscv}).

\begin{lemma}
\label{sfrac}Let $u\in C_{0}^{0,s}(\overline{\Omega})$ be extended as zero
outside $\Omega.$ If $u\in W^{s,q}(\Omega)$ for some $q>1,$ then $u\in
W_{0}^{s,p}(\Omega)$ for all $p\geq q$ and%
\begin{equation}
\lim_{p\rightarrow\infty}\left[  u\right]  _{s,p}=\left\vert u\right\vert
_{s}. \label{ssfrac}%
\end{equation}

\end{lemma}

\begin{proof}
First, note that the inequality
\[
\left\vert u(x)-u(y)\right\vert \leq\left\vert u\right\vert _{s}\left\vert
x-y\right\vert ^{s}%
\]
is valid for all $x,y\in\mathbb{R}^{N},$ not only for those $x,y\in
\overline{\Omega}.$ In fact, this is obvious when $x,y\in\mathbb{R}%
^{N}\setminus\overline{\Omega}.$ Now, if $x\in\Omega$ and $y\in\mathbb{R}%
^{N}\setminus\overline{\Omega}$ then take $y_{1}\in\partial\Omega$ such that
$\left\vert x-y_{1}\right\vert \leq\left\vert x-y\right\vert $ (such $y_{1}$
can be taken on the straight line connecting $x$ to $y$). Since $u(y)=u(y_{1}%
)=0,$ we have
\[
\left\vert u(x)-u(y)\right\vert =\left\vert u(x)\right\vert =\left\vert
u(x)-u(y_{1})\right\vert \leq\left\vert u\right\vert _{s}\left\vert
x-y_{1}\right\vert ^{s}\leq\left\vert u\right\vert _{s}\left\vert
x-y\right\vert ^{s}.
\]

For each $p>q$ we have
\[
\left[  u\right]  _{s,p}^{p}=\int_{\mathbb{R}^{N}}\int_{\mathbb{R}^{N}}%
\frac{\left\vert u(x)-u(y)\right\vert ^{p-q}}{\left\vert x-y\right\vert
^{s(p-q)}}\frac{\left\vert u(x)-u(y)\right\vert ^{q}}{\left\vert
x-y\right\vert ^{N+sq}}\mathrm{d}x\mathrm{d}y\leq(\left\vert u\right\vert
_{s})^{(p-q)}\left[  u\right]  _{s,q}^{q}.
\]
Thus, $u\in W_{0}^{s,p}(\Omega)$ and
\begin{equation}
\limsup_{p\rightarrow\infty}\left[  u\right]  _{s,p}\leq\lim_{p\rightarrow
\infty}\left\vert u\right\vert _{s}^{(p-q)/p}\left[  u\right]  _{s,q}%
^{q/p}=\left\vert u\right\vert _{s}. \label{up5}%
\end{equation}

Now, noticing that (by Fatou's lemma)
\[
\int_{\Omega}\int_{\Omega}\left(  \frac{\left\vert u(x)-u(y)\right\vert
}{\left\vert x-y\right\vert ^{s}}\right)  ^{q}\mathrm{d}x\mathrm{d}%
y\leq\liminf_{p\rightarrow\infty}\int_{\Omega}\int_{\Omega}\left(
\frac{\left\vert u(x)-u(y)\right\vert }{\left\vert x-y\right\vert ^{\frac
{N}{p}+s}}\right)  ^{q}\mathrm{d}x\mathrm{d}y
\]
and (by H\"{o}lder's inequality)
\begin{align*}
\int_{\Omega}\int_{\Omega}\left(  \frac{\left\vert u(x)-u(y)\right\vert
}{\left\vert x-y\right\vert ^{\frac{N}{p}+s}}\right)  ^{q}\mathrm{d}%
x\mathrm{d}y  &  \leq\left\vert \Omega\right\vert ^{2(1-\frac{q}{p})}\left(
\int_{\Omega}\int_{\Omega}\left(  \frac{\left\vert u(x)-u(y)\right\vert
}{\left\vert x-y\right\vert ^{\frac{N}{p}+s}}\right)  ^{p}\mathrm{d}%
x\mathrm{d}y\right)  ^{\frac{q}{p}}\\
&  \leq\left\vert \Omega\right\vert ^{2(1-\frac{q}{p})}\left[  u\right]
_{s,p}^{q},
\end{align*}
we obtain
\[
\left(  \int_{\Omega}\int_{\Omega}\left(  \frac{\left\vert
u(x)-u(y)\right\vert }{\left\vert x-y\right\vert ^{s}}\right)  ^{q}%
\mathrm{d}x\mathrm{d}y\right)  ^{\frac{1}{q}}\leq\left\vert \Omega\right\vert
^{2/q}\liminf_{p\rightarrow\infty}\left[  u\right]  _{s,p}.
\]

Hence, taking into account that%
\[
\left\vert u\right\vert _{s}=\lim_{q\rightarrow\infty}\left(  \int_{\Omega
}\int_{\Omega}\left(  \frac{\left\vert u(x)-u(y)\right\vert }{\left\vert
x-y\right\vert ^{s}}\right)  ^{q}\mathrm{d}x\mathrm{d}y\right)  ^{\frac{1}{q}}%
\]
we arrive at%
\[
\left\vert u\right\vert _{s}\leq\lim_{q\rightarrow\infty}\left\vert
\Omega\right\vert ^{2/q}\left(  \liminf_{p\rightarrow\infty}\left[  u\right]
_{s,p}\right)  =\liminf_{p\rightarrow\infty}\left[  u\right]  _{s,p}.
\]

This estimate combined with (\ref{up5}) leads us to (\ref{ssfrac}).
\end{proof}

It is known (see \cite[Theorem 8.2]{Guide}) that if $p>\dfrac{N}{s}$ then
there exists of a positive constant $C$ such that
\begin{equation}
\left\Vert u\right\Vert _{C^{0,\beta}(\overline{\Omega})}\leq C\left[
u\right]  _{s,p}\quad\forall\,u\in W_{0}^{s,p}(\Omega), \label{morrey}%
\end{equation}
where $\beta:=s-\dfrac{N}{p}\in(0,1).$ As pointed out in \cite[Remark
2.2]{FPL} the constant $C$ in (\ref{morrey}) can be chosen uniform with
respect to $p.$

We remark that the family of positive numbers $\left\{  \sqrt[p]{\Lambda_{p}%
}\right\}  _{p>1}$ is bounded. Indeed, combining (\ref{Lamfin}) with the
previous lemma we obtain
\[
\limsup_{p\rightarrow\infty}\sqrt[p]{\Lambda_{p}}\leq\left\vert \xi\right\vert
_{s}.
\]

The next lemma, where $\operatorname{Id}$ stands for the identity function, is
extracted of the proof of \cite[Lemma 3.2]{LiChu}. It helps us to overcome the
fact that $C_{c}^{\infty}(\Omega)$ is not dense in $C_{0}^{0,s}(\overline
{\Omega}).$

\begin{lemma}
[{see \cite[Lemma 3.2]{LiChu}}]\label{litchu}Let $\Omega\subset\mathbb{R}^{N}$
be a Lipschitz bounded domain. There exist $\phi\in C_{c}^{\infty}%
(\mathbb{R}^{N},\mathbb{R}^{N})$ and $0<\tau_{0}<(\left\vert \phi\right\vert
_{1})^{-1}$ such that, for each $0\leq\tau\leq\tau_{0},$ the map
\[
\Phi_{\tau}:=\operatorname{Id}+\tau\phi:\mathbb{R}^{N}\rightarrow
\mathbb{R}^{N}%
\]
is a diffeomorphism satisfying

\begin{enumerate}
\item $\Phi_{\tau}(\overline{\Omega})\subset\subset\Omega,$

\item $\Phi_{\tau}\rightarrow\operatorname{Id}$ and $(\Phi_{\tau}%
)^{-1}\rightarrow\operatorname{Id}$ as $\tau\rightarrow0^{+}$ uniformly on
$\mathbb{R}^{N}$,

\item $\left\vert (\Phi_{\tau})^{-1}(x)-(\Phi_{\tau})^{-1}(y)\right\vert
\leq\dfrac{\left\vert x-y\right\vert }{1-\tau\left\vert \phi\right\vert _{1}%
}.$
\end{enumerate}
\end{lemma}

\begin{lemma}
\label{dens}Let $u\in C_{0}^{0,s}(\overline{\Omega})$ be a nonnegative
function extended as zero outside $\Omega.$ There exists a sequence of
nonnegative functions $\left\{  u_{k}\right\}  _{k\in\mathbf{N}}\subset
C_{0}^{0,s}(\overline{\Omega})\cap W_{0}^{s,p}(\Omega),$ for all $p>1,$
converging uniformly to $u$ in $\overline{\Omega}$ and such that%
\[
\limsup_{k\rightarrow\infty}\left\vert u_{k}\right\vert _{s}\leq\left\vert
u\right\vert _{s}.
\]

\end{lemma}

\begin{proof}
For each $k\in\mathbb{N}$ let $\Psi_{k}$ denote the inverse of $\Phi_{1/k},$
given by Lemma \ref{litchu}, and set%
\[
\Omega_{k}:=\Phi_{1/k}(\overline{\Omega}).
\]
Since $\Omega_{k}\subset\subset\Omega$ there exists $U_{k},$ a subdomain of
$\Omega,$ such that
\[
\overline{\Omega_{k}}\subset U_{k}\subset\overline{U_{k}}\subset\Omega.
\]

Let $\eta\in C^{\infty}(\mathbb{R}^{N})$ be a standard convolution kernel:
$\eta(z)>0$ if $\left\vert z\right\vert <1,$ $\eta(z)=0$ if $\left\vert
z\right\vert \geq1$ and $\int_{\left\vert z\right\vert \leq1}\phi
(z)\mathrm{d}z=1.$

Define the function%
\[
u_{k}=(u\circ\Psi_{k})\ast\eta_{k}\in C^{\infty}(\mathbb{R}^{N}),
\]
where
\[
\eta_{k}(x):=(\epsilon_{k})^{-N}\eta(\frac{x}{\epsilon_{k}}),\quad
x\in\mathbb{R}^{N}%
\]
and $\epsilon_{k}<\operatorname{dist}(\Omega_{k},\partial U_{k}).$ Note that
$\epsilon_{k}\rightarrow0.$

Since
\[
B(x,\epsilon_{k})\subset\mathbb{R}^{N}\setminus\Omega_{k}\quad\forall
\,x\in\mathbb{R}^{N}\setminus U_{k},
\]
we have%
\[
\Psi_{k}(B(x,\epsilon_{k}))\subset\mathbb{R}^{N}\setminus\Omega\quad
\forall\,x\in\mathbb{R}^{N}\setminus U_{k}.
\]
Hence, observing that
\[
u_{k}(x)=\int_{\mathbb{R}^{N}}\eta_{k}(x-z)u(\Psi_{k}(z))\mathrm{d}%
z=\int_{B(0,1)}\eta(z)u(\psi_{k}(x-\epsilon_{k}z))\mathrm{d}z\quad
\forall\,x\in\mathbb{R}^{N}%
\]
and that
\[
\left\vert x-\epsilon_{k}z-x\right\vert \leq\epsilon_{k}\quad\forall\,z\in
B(0,1)
\]
we conclude that
\[
u_{k}(x)=0\quad\forall\,x\in\mathbb{R}^{N}\setminus U_{k}.
\]
Therefore, $u_{k}\in C_{c}^{\infty}(\Omega)\subset W_{0}^{1,p}(\Omega)$ for
all $p>1.$

Now, let $x,y\in\overline{\Omega}$ be fixed. According to item $3$ of Lemma
\ref{litchu}
\begin{align*}
\left\vert u_{k}(x)-u_{k}(y)\right\vert  &  \leq\int_{B(0,1)}\eta(z)\left\vert
u(\Psi_{k}(x-\epsilon_{k}z))-u(\Psi_{k}(y-\epsilon_{k}z))\right\vert
\mathrm{d}z\\
&  \leq\left\vert u\right\vert _{s}\int_{B(0,1)}\eta(z)\left\vert \Psi
_{k}(x-\epsilon_{k}z)-\Psi_{k}(y-\epsilon_{k}z))\right\vert ^{s}\mathrm{d}z\\
&  \leq\dfrac{\left\vert u\right\vert _{s}}{(1-(1/k)\left\vert \phi\right\vert
_{1})^{s}}\int_{B(0,1)}\eta(z)\left\vert x-y\right\vert ^{s}\mathrm{d}z\\
&  =\dfrac{\left\vert u\right\vert _{s}}{(1-(1/k)\left\vert \phi\right\vert
_{1})^{s}}\left\vert x-y\right\vert ^{s}.
\end{align*}

It follows that $u_{k}\in C_{0}^{0,s}(\overline{\Omega})$ and
\[
\limsup_{k\rightarrow\infty}\left\vert u_{k}\right\vert _{s}\leq
\lim_{k\rightarrow\infty}\dfrac{\left\vert u\right\vert _{s}}%
{(1-(1/k)\left\vert \phi\right\vert _{1})^{s}}=\left\vert u\right\vert _{s}.
\]
Consequently, up to a subsequence, $u_{k}\rightarrow\widetilde{u}\in
C(\overline{\Omega})$ uniformly in $\overline{\Omega}.$ Hence, $\widetilde
{u}=u$ since item $2$ of Lemma \ref{litchu} implies that%
\[
\lim_{k\rightarrow\infty}u_{k}(x)=\int_{B(0,1)}\eta(z)u(\lim_{k\rightarrow
\infty}\Psi_{k}(x-\epsilon_{k}z))\mathrm{d}z=u(x)\int_{B(0,1)}\eta
(z)\mathrm{d}z=u(x).
\]

\end{proof}

\begin{theorem}
\label{conv1}Let $p_{n}\rightarrow\infty.$ Up to a subsequence, $\left\{
u_{p_{n}}\right\}  _{n\in\mathbb{N}}$ converges uniformly to a nonnegative
function $u_{\infty}\in C_{0}^{0,s}(\overline{\Omega})$ such that
\[
\left\vert u_{\infty}\right\vert _{s}=\lim_{n\rightarrow\infty}\sqrt[p_{n}%
]{\Lambda_{p_{n}}}.
\]
Furthermore,%
\begin{equation}
v_{s}=(k_{\infty})^{-1}u_{\infty} \label{vsu}%
\end{equation}
where
\begin{equation}
k_{\infty}:=k(u_{\infty})=\exp\left(  \int_{\Omega}(\log\left\vert u_{\infty
}\right\vert )\omega\mathrm{d}x\right)  \geq1. \label{kinfty}%
\end{equation}

\end{theorem}

\begin{proof}
Let $p_{0}>\dfrac{N}{s}$ be fixed and take $\beta_{0}=s-\frac{N}{p_{0}}.$ For
each $(x,y)\in\Omega\times\Omega,$ with $x\not =y,$ we obtain from
(\ref{morrey})
\begin{align*}
\frac{\left\vert u_{p}(x)-u_{p}(y)\right\vert }{\left\vert x-y\right\vert
^{s-\frac{N}{p_{0}}}}  &  =\frac{\left\vert u_{p}(x)-u_{p}(y)\right\vert
}{\left\vert x-y\right\vert ^{s-\frac{N}{p}}}\left\vert x-y\right\vert
^{N(\frac{1}{p_{0}}-\frac{1}{p})}\\
&  \leq C\left[  u_{p}\right]  _{s,p}\operatorname{diam}(\Omega)^{N(\frac
{1}{p_{0}}-\frac{1}{p})},\quad\forall\,p\geq p_{0},
\end{align*}
where $C$ is uniform with respect to $p$ and $\operatorname{diam}(\Omega)$ is
the diameter of $\Omega.$ Hence, in view of (\ref{up1}) and (\ref{Lamfin}) the
family $\left\{  u_{p}\right\}  _{p\geq p_{0}}$ is bounded in $C_{0}%
^{0,\beta_{0}}(\overline{\Omega}),$ implying that, up to a subsequence,
$u_{p_{n}}\rightarrow u_{\infty}\in C(\overline{\Omega})$ uniformly in
$\overline{\Omega}.$ Of course, the limit function $u_{\infty}$ is nonnegative
in $\Omega$ and vanishes on $\partial\Omega.$

Letting $n\rightarrow\infty$ in the inequality (which follows from
(\ref{morrey}))
\[
\frac{\left\vert u_{p_{n}}(x)-u_{p_{n}}(y)\right\vert }{\left\vert
x-y\right\vert ^{s-\frac{N}{p_{n}}}}\leq C\left[  u_{p_{n}}\right]  _{s,p_{n}%
}=C\sqrt[p_{n}]{\Lambda_{p_{n}}}%
\]
and taking (\ref{Lamfin}) into account we conclude that $u_{\infty}\in
C_{0}^{0,s}(\overline{\Omega}).$

Up to another subsequence, we can assume that
\[
\sqrt[p_{n}]{\Lambda_{p_{n}}}\rightarrow L.
\]

Let $q>\dfrac{N}{s}$ be fixed. By Fatou's Lemma and H\"{o}lder's inequality,
\begin{align*}
\int_{\Omega}\int_{\Omega}\left(  \frac{\left\vert u_{\infty}(x)-u_{\infty
}(y)\right\vert }{\left\vert x-y\right\vert ^{s}}\right)  ^{q}\mathrm{d}%
x\mathrm{d}y  &  \leq\liminf_{n\rightarrow\infty}\int_{\Omega}\int_{\Omega
}\left(  \frac{\left\vert u_{p_{n}}(x)-u_{p_{n}}(y)\right\vert }{\left\vert
x-y\right\vert ^{\frac{N}{p_{n}}+s}}\right)  ^{q}\mathrm{d}x\mathrm{d}y\\
&  \leq\liminf_{n\rightarrow\infty}\left\vert \Omega\right\vert ^{2(1-\frac
{q}{p_{n}})}\left(  \int_{\Omega}\int_{\Omega}\left(  \frac{\left\vert
u_{p_{n}}(x)-u_{p_{n}}(y)\right\vert }{\left\vert x-y\right\vert ^{\frac
{N}{p_{n}}+s}}\right)  ^{p_{n}}\mathrm{d}x\mathrm{d}y\right)  ^{\frac{q}%
{p_{n}}}\\
&  \leq\left\vert \Omega\right\vert ^{2}\liminf_{n\rightarrow\infty}\left[
u_{p_{n}}\right]  _{s,p_{n}}^{q}=\left\vert \Omega\right\vert ^{2}%
\lim_{n\rightarrow\infty}(\sqrt[p_{n}]{\Lambda_{p_{n}}})^{q}=\left\vert
\Omega\right\vert ^{2}L^{q}.
\end{align*}

Therefore,%
\begin{equation}
\left\vert u_{\infty}\right\vert _{s}=\lim_{q\rightarrow\infty}\left(
\int_{\Omega}\int_{\Omega}\left(  \frac{\left\vert u_{\infty}(x)-u_{\infty
}(y)\right\vert }{\left\vert x-y\right\vert ^{s}}\right)  ^{q}\mathrm{d}%
x\mathrm{d}y\right)  ^{1/q}\leq\lim_{q\rightarrow\infty}\left\vert
\Omega\right\vert ^{\frac{2}{q}}L=L. \label{up7}%
\end{equation}

To prove that $k_{\infty}\geq1$ we first note that%
\[
\lim_{t\rightarrow0^{+}}\left(
{\displaystyle\int_{\Omega}}
\left\vert u_{p_{n}}\right\vert ^{t}\omega\mathrm{d}x\right)  ^{\frac{1}{t}%
}=\inf_{0<t<1}\left(
{\displaystyle\int_{\Omega}}
\left\vert u_{p_{n}}\right\vert ^{t}\omega\mathrm{d}x\right)  ^{\frac{1}{t}%
}\leq\left(
{\displaystyle\int_{\Omega}}
\left\vert u_{p_{n}}\right\vert ^{\epsilon}\omega\mathrm{d}x\right)
^{\frac{1}{\epsilon}}\quad\forall\,\epsilon\in(0,1).
\]
Consequently,
\[
1=k(u_{p_{n}})=\lim_{t\rightarrow0^{+}}\left(
{\displaystyle\int_{\Omega}}
\left\vert u_{p_{n}}\right\vert ^{t}\omega\mathrm{d}x\right)  ^{\frac{1}{t}%
}\leq\left(
{\displaystyle\int_{\Omega}}
\left\vert u_{p_{n}}\right\vert ^{\epsilon}\omega\mathrm{d}x\right)
^{\frac{1}{\epsilon}}.
\]
The uniform convergence $u_{p_{n}}\rightarrow u_{\infty}$ then yields%
\[
1\leq\lim_{n\rightarrow\infty}\left(
{\displaystyle\int_{\Omega}}
\left\vert u_{p_{n}}\right\vert ^{\epsilon}\omega\mathrm{d}x\right)
^{\frac{1}{\epsilon}}=\left(
{\displaystyle\int_{\Omega}}
\left\vert u_{\infty}\right\vert ^{\epsilon}\omega\mathrm{d}x\right)
^{\frac{1}{\epsilon}}.
\]

Therefore,
\[
k_{\infty}=k(u_{\infty})=\lim_{\epsilon\rightarrow0^{+}}\left(  \int_{\Omega
}\left\vert u_{\infty}\right\vert ^{\epsilon}\omega\mathrm{d}x\right)
^{\frac{1}{\epsilon}}\geq1.
\]
It follows that $(k_{\infty})^{-1}u_{\infty}\in\mathcal{M}_{s},$ so that
\begin{equation}
\mu_{s}\leq\left\vert (k_{\infty})^{-1}u_{\infty}\right\vert _{s}=(k_{\infty
})^{-1}\left\vert u_{\infty}\right\vert _{s}. \label{mumin2}%
\end{equation}

In the next step we prove that
\begin{equation}
\int_{\Omega}\frac{u}{u_{\infty}}\omega\mathrm{d}x\leq\frac{\left\vert
u\right\vert _{s}}{L}\quad\forall\,u\in C_{0}^{0,s}(\overline{\Omega}).
\label{mumin}%
\end{equation}

According to Lemma \ref{dens} there exists a sequence of nonnegative functions
$\left\{  u_{k}\right\}  _{k\in\mathbf{N}}\subset C_{0}^{0,s}(\overline
{\Omega})\cap W_{0}^{s,p}(\Omega),$ for all $p>1,$ converging uniformly to $u$
in $C(\overline{\Omega})$ and such that%
\[
\limsup_{k\rightarrow\infty}\left\vert u_{k}\right\vert _{s}\leq\left\vert
u\right\vert _{s}.
\]

Since $u_{p}$ is the weak solution of (\ref{extremal}) and $\Lambda
_{p}=\left[  u_{p}\right]  _{s,p}^{p}$ we use H\"{o}lder's inequality to get
\[
\Lambda_{p}\int_{\Omega}\frac{u_{k}}{u_{p}}\omega\mathrm{d}x=\left\langle
(-\Delta_{p})^{s}u_{p},u_{k}\right\rangle \leq\left[  u_{p}\right]
_{s,p}^{p-1}\left[  u_{k}\right]  _{s,p}=(\Lambda_{p})^{\frac{p-1}{p}}\left[
u_{k}\right]  _{s,p}.
\]

It follows that%
\[
\sqrt[p_{n}]{\Lambda_{p_{n}}}\int_{\Omega}\frac{u_{k}}{u_{p_{n}}}%
\omega\mathrm{d}x\leq\left[  u_{k}\right]  _{s,p_{n}}.
\]

Combining Fatou's lemma with the uniform convergence $u_{p_{n}}\rightarrow
u_{\infty}$ and the Lemma \ref{sfrac} we obtain
\[
L\int_{\Omega}\frac{u_{k}}{u_{\infty}}\omega\mathrm{d}x\leq L\liminf
_{n\rightarrow\infty}\int_{\Omega}\frac{u_{k}}{u_{p_{n}}}\omega\mathrm{d}%
x\leq\liminf_{n\rightarrow\infty}\left[  u_{k}\right]  _{s,p_{n}}=\left\vert
u_{k}\right\vert _{s},
\]
that is,%
\[
L\int_{\Omega}\frac{u_{k}}{u_{\infty}}\omega\mathrm{d}x\leq\left\vert
u_{k}\right\vert _{s}.
\]

Letting $k\rightarrow\infty$ and applying Fatou's lemma again we arrive at
(\ref{mumin}):
\[
L\int_{\Omega}\frac{u}{u_{\infty}}\omega\mathrm{d}x\leq L\liminf
_{k\rightarrow\infty}\int_{\Omega}\frac{u_{k}}{u_{\infty}}\omega
\mathrm{d}x\leq\liminf_{k\rightarrow\infty}\left\vert u_{k}\right\vert
_{s}\leq\left\vert u\right\vert _{s}.
\]

Taking $u=u_{\infty}$ in (\ref{mumin}) we obtain%
\[
L\leq\left\vert u_{\infty}\right\vert _{s}%
\]
and combining this with (\ref{up7}) we conclude that
\begin{equation}
L=\left\vert u_{\infty}\right\vert _{s}. \label{mumin1}%
\end{equation}

Now, let $0\leq u\in\mathcal{M}_{s}$ be fixed. Then (\ref{jensen1}) yields
\begin{align*}
-\int_{\Omega}(\log u_{\infty})\omega\mathrm{d}x  &  =\int_{\Omega}(\log
u)\omega\mathrm{d}x-\int_{\Omega}(\log u_{\infty})\omega\mathrm{d}x\\
&  =\int_{\Omega}(\log(\frac{u}{u_{\infty}}))\omega\mathrm{d}x\leq\log\left(
\int_{\Omega}\frac{u}{u_{\infty}}\omega\mathrm{d}x\right)  .
\end{align*}
Hence, (\ref{mumin}) and (\ref{mumin1}) imply that
\begin{equation}
(k_{\infty})^{-1}\leq\int_{\Omega}\frac{u}{u_{\infty}}\omega\mathrm{d}%
x\leq\frac{\left\vert u\right\vert _{s}}{\left\vert u_{\infty}\right\vert
_{s}}\quad\mathrm{whenever}\quad0\leq u\in\mathcal{M}_{s}. \label{mumin4}%
\end{equation}

Combining these estimates at $u=v_{s}$ with (\ref{mumin2}) we obtain
\[
(k_{\infty})^{-1}\leq\int_{\Omega}\frac{v_{s}}{u_{\infty}}\omega
\mathrm{d}x\leq\frac{\left\vert v_{s}\right\vert _{s}}{\left\vert u_{\infty
}\right\vert _{s}}=\frac{\mu_{s}}{\left\vert u_{\infty}\right\vert _{s}}%
\leq(k_{\infty})^{-1},
\]
which leads us to conclude that%
\[
\mu_{s}=\left\vert (k_{\infty})^{-1}u_{\infty}\right\vert _{s}\quad
\mathrm{and}\quad(k_{\infty})^{-1}=\int_{\Omega}\frac{v_{s}}{u_{\infty}}%
\omega\mathrm{d}x.
\]

Since $v_{s}$ is the only nonnegative minimizer of $\left\vert \cdot
\right\vert _{s}$ on $\mathcal{M}_{s}$ we get (\ref{vsu}).
\end{proof}

\begin{corollary}
The following inequalities hold%
\begin{equation}
k(u)\leq\int_{\Omega}\frac{\left\vert u\right\vert }{v_{s}}\omega
\mathrm{d}x\leq\frac{\left\vert u\right\vert _{s}}{\mu_{s}}\quad\forall\,u\in
C_{0}^{0,s}(\overline{\Omega}). \label{corol}%
\end{equation}

\end{corollary}

\begin{proof}
Since we already know that $L=\left\vert u_{\infty}\right\vert _{s}$ and
$u_{\infty}=k_{\infty}v_{s}$ the second inequality in (\ref{corol}) follows
from (\ref{mumin}), with $u$ replaced with $w=\left\vert u\right\vert $ (note
that $\left\vert w\right\vert _{s}\leq\left\vert u\right\vert _{s}$). The
first inequality in (\ref{corol}) is obvious when $k(u)=0$ and, when $k(u)>0,$
it follows from the first inequality in (\ref{mumin4}), with $w=(k(u))^{-1}%
\left\vert u\right\vert \in\mathcal{M}_{s}.$
\end{proof}

\begin{remark}
\label{kleq0}In contrast with what happens in similar problems driven by the
standard $p$-Laplacian, we are not able to prove that $u_{\infty}\in
W_{0}^{s,q}(\Omega)$ for some $q>1.$ Such a property would guarantee that
$u_{\infty}=v_{s}$ and, consequently,
\[
\lim_{p\rightarrow\infty}u_{p}=v_{s}%
\]
(that is, $v_{s}$ would be the only limit point of the family $\left\{
u_{p}\right\}  _{p>1},$ as $p\rightarrow\infty$). Indeed, if $u_{\infty}\in
W_{0}^{s,q}(\Omega)$ for some $q>1$ then, according to Lemma \ref{sfrac},
$u_{\infty}\in W_{0}^{s,p_{n}}(\Omega)$ for all $n$ sufficiently large (such
that $p_{n}\geq q$) and%
\[
\lim_{n\rightarrow\infty}\left[  u_{\infty}\right]  _{s,p_{n}}=\left\vert
u_{\infty}\right\vert _{s}.
\]
Hence, proceeding as in the proof of Theorem \ref{conv1}, we would arrive at
\[
1\leq k_{\infty}\leq\int_{\Omega}\frac{u_{\infty}}{u_{p_{n}}}\omega
\mathrm{d}x\leq\frac{\left[  u_{\infty}\right]  _{s,p_{n}}}{\sqrt[p_{n}%
]{\Lambda_{p_{n}}}}.
\]
Since $\lim_{n\rightarrow\infty}\left[  u_{\infty}\right]  _{s,p_{n}}%
=\lim_{n\rightarrow\infty}\sqrt[p_{n}]{\Lambda_{p_{n}}}=\left\vert u_{\infty
}\right\vert _{s}$ we would conclude that $k_{\infty}=1$ and $u_{\infty}%
=v_{s}.$
\end{remark}

\section{The limit problem\label{sec3}}

For a matter of compatibility with the viscosity approach we add the
hypotheses of continuity and strict positiveness to the weight $\omega$. So,
we assume in this section that
\[
\omega\in C(\Omega)\cap L^{r}(\Omega),\,r>1,\quad\omega>0\quad\mathrm{in}%
\quad\Omega,\quad\mathrm{and}\quad\int_{\Omega}\omega\mathrm{d}x=1.
\]
Note that such $\omega$ satisfies the hypotheses of Theorem \ref{teomana}.

For $1<p<\infty$ we write the $s$-fractional $p$-Laplacian, in its integral
version, as $\left(  -\Delta_{p}\right)  ^{s}=-\mathcal{L}_{p}$ where
\begin{equation}
(\mathcal{L}_{p}u)(x):=2\int_{\mathbb{R}^{N}}\dfrac{\left\vert
u(y)-u(x)\right\vert ^{p-2}(u(y)-u(x))}{\left\vert y-x\right\vert ^{N+sp}%
}\mathrm{d}y. \label{eqp}%
\end{equation}
Corresponding to the case $p=\infty$ we define operator $\mathcal{L}_{\infty}$
by
\begin{equation}
\mathcal{L}_{\infty}:=\mathcal{L}_{\infty}^{+}+\mathcal{L}_{\infty}^{-},
\label{linf}%
\end{equation}
where%
\begin{equation}
\left(  \mathcal{L}_{\infty}^{+}u\right)  (x):=\sup_{y\in\mathbb{R}%
^{N}\setminus\left\{  x\right\}  }\frac{u(y)-u(x)}{\left\vert y-x\right\vert
^{s}}\quad\mathrm{and}\quad\left(  \mathcal{L}_{\infty}^{-}u\right)
(x):=\inf_{y\in\mathbb{R}^{N}\setminus\left\{  x\right\}  }\frac
{u(y)-u(x)}{\left\vert y-x\right\vert ^{s}}. \label{op+-}%
\end{equation}

In the sequel we consider, in the viscosity sense, the problem%
\begin{equation}
\left\{
\begin{array}
[c]{lll}%
\mathcal{L}u=0 & \mathrm{in} & \Omega\\
u=0 & \mathrm{in} & \mathbb{R}^{N}\setminus\Omega,
\end{array}
\right.  \label{eqvisc}%
\end{equation}
where either $\mathcal{L}u=\mathcal{L}_{p}u+\Lambda_{p}u^{-1}\omega,$ with
$1<p<\infty,$ or%
\[
\mathcal{L}u=\mathcal{L}_{\infty}u\quad\mathrm{or}\quad\mathcal{L}%
u=\mathcal{L}_{\infty}^{-}u+\left\vert u_{\infty}\right\vert _{s}.
\]

We recall some definitions related to the viscosity approach for the problem
(\ref{eqvisc}).

\begin{definition}
Let $u\in C(\mathbb{R}^{N})$ such that $u>0$ in $\Omega$ and $u=0$ in
$\mathbb{R}^{N}\setminus\Omega.$ We say that $u$ is a viscosity supersolution
of the equation (\ref{eqvisc}) if
\[
(\mathcal{L}\varphi)(x_{0})\leq0
\]
for all pair $\left(  x_{0},\varphi\right)  \in\Omega\times C_{0}%
^{1}(\mathbb{R}^{N})$ satisfying
\[
\varphi(x_{0})=u(x_{0})\quad\mathrm{and}\quad\varphi(x)\leq u(x)\quad
\forall\,x\in\mathbb{R}^{N}.
\]

Analogously, we say that $u$ is a viscosity subsolution of (\ref{eqvisc}) if
\[
(\mathcal{L}\varphi)(x_{0})\geq0
\]
for all pair $\left(  x_{0},\varphi\right)  \in\Omega\times C_{0}%
^{1}(\mathbb{R}^{N})$ satisfying%
\[
\varphi(x_{0})=u(x_{0})\quad\mathrm{and}\quad\varphi(x)\geq u(x)\quad
\forall\,x\in\mathbb{R}^{N}.
\]

We say that $u$ is a viscosity solution of (\ref{eqvisc}) if it is
simultaneously a subsolution and a supersolution of (\ref{eqvisc}).
\end{definition}

The next lemma can be proved by following, step by step, the proof of
Proposition 11 of \cite{LindLind}.

\begin{lemma}
Let $u\in W_{0}^{s,p}(\Omega)\cap C(\overline{\Omega})$ be a positive weak
solution of (\ref{extremal}). Then $u$ is a viscosity solution of
\begin{equation}
\left\{
\begin{array}
[c]{lll}%
\mathcal{L}_{p}u+\Lambda_{p}u^{-1}\omega=0 & \mathrm{in} & \Omega\\
u=0 & \mathrm{in} & \mathbb{R}^{N}\setminus\Omega.
\end{array}
\right.  \label{Lpvisc}%
\end{equation}

\end{lemma}

Our main result in this section is the following, where $u_{\infty}\in
C_{0}^{0,s}(\overline{\Omega})$ is the function given by Theorem \ref{conv1}.

\begin{theorem}
\label{viscu}The function $u_{\infty}\in C_{0}^{0,s}(\overline{\Omega}),$
extended as zero outside $\Omega,$ is both a viscosity supersolution of the
problem
\begin{equation}
\left\{
\begin{array}
[c]{lll}%
\mathcal{L}_{\infty}u=0 & \mathrm{in} & \Omega\\
u=0 & \mathrm{in} & \mathbb{R}^{N}\setminus\Omega
\end{array}
\right.  \label{viscinf}%
\end{equation}
and a viscosity solution of the problem%
\begin{equation}
\left\{
\begin{array}
[c]{lll}%
\mathcal{L}_{\infty}^{-}u+\left\vert u_{\infty}\right\vert _{s}=0 &
\mathrm{in} & \Omega\\
u=0 & \mathrm{in} & \mathbb{R}^{N}\setminus\Omega.
\end{array}
\right.  \label{visc}%
\end{equation}

Moreover, $u_{\infty}$ is strictly positive in $\Omega$ and the only
minimizers of $\left\vert \cdot\right\vert _{s}$ on $\mathcal{M}_{s}$ are
\begin{equation}
-v_{s}\quad\mathrm{and}\quad v_{s}. \label{+-vs}%
\end{equation}

\end{theorem}

\begin{proof}
We begin by proving that $u_{\infty}$ is a viscosity supersolution of
(\ref{visc}). For this, let us fix $\left(  x_{0},\varphi\right)  \in
\Omega\times C_{0}^{1}(\mathbb{R}^{N})$ satisfying
\begin{equation}
\varphi(x_{0})=u_{\infty}(x_{0})\quad\mathrm{and}\quad\varphi(x)\leq
u_{\infty}(x)\quad\forall\,x\in\mathbb{R}^{N}. \label{visc3}%
\end{equation}

Without loss of generality we can assume that
\[
\varphi(x)<u_{\infty}(x)\quad\forall\,x\in\mathbb{R}^{N},
\]
what allows us to assure that $u_{p_{n}}-\varphi$ assumes its minimum value at
a point $x_{n},$ with $x_{n}\rightarrow x_{0}.$

Let $c_{n}:=u_{p_{n}}(x_{n})-\varphi(x_{n}).$ Of course, $c_{n}\rightarrow0$
(due to the uniform convergence $u_{p_{n}}\rightarrow u_{\infty}$). By
construction,
\[
\varphi(x_{n})+c_{n}=u_{p_{n}}(x_{n})\quad\mathrm{and}\quad\varphi
(x)+c_{n}\leq u_{p_{n}}(x)\quad\forall\,x\in\mathbb{R}^{N}.
\]

According to the previous lemma, $u_{p}$ is a viscosity supersolution of
(\ref{Lpvisc}) since it is a viscosity solution of the same problem.
Therefore,%
\[
(\mathcal{L}_{p_{n}}\varphi)(x_{n})+\Lambda_{p_{n}}\frac{\omega(x_{n}%
)}{u_{p_{n}}(x_{n})}=(\mathcal{L}_{p_{n}}(\varphi+c_{n}))(x_{n})+\Lambda
_{p_{n}}\frac{\omega(x_{n})}{\varphi(x_{n})+c_{n}}\leq0,
\]
an inequality that can be rewritten as
\[
A_{n}^{p_{n}-1}+C_{n}^{p_{n}-1}\leq B_{n}^{p_{n}-1}%
\]
where%
\[
A_{n}^{p_{n}-1}=2\int_{\mathbb{R}^{N}}\dfrac{\left\vert \varphi(y)-\varphi
(x_{n})\right\vert ^{p_{n}-2}(\varphi(y)-\varphi(x_{n}))^{+}}{\left\vert
y-x\right\vert ^{N+sp_{n}}}\mathrm{d}y\geq0,
\]%
\[
B_{n}^{p_{n}-1}=2\int_{\mathbb{R}^{N}}\dfrac{\left\vert \varphi(y)-\varphi
(x_{n})\right\vert ^{p_{n}-2}(\varphi(y)-\varphi(x_{n}))^{-}}{\left\vert
y-x\right\vert ^{N+sp_{n}}}\mathrm{d}y\geq0,
\]
and%
\[
C_{n}^{p_{n}-1}=\Lambda_{p_{n}}\frac{\omega(x_{n})}{u_{p_{n}}(x_{n})}>0.
\]
(Here, $a^{+}:=\max\left\{  a,0\right\}  $ and $a^{-}:=\max\left\{
-a,0\right\}  ,$ so that $a=a^{+}-a^{-}.$)

According to Lemma 6.1 of \cite{FPL}, which was adapted from Lemma 6.5 of
\cite{Chamb}, we have
\[
\lim_{n\rightarrow\infty}A_{n}=\left(  \mathcal{L}_{\infty}^{+}\varphi\right)
(x_{0})\quad\mathrm{and}\quad\lim_{n\rightarrow\infty}B_{n}=-\left(
\mathcal{L}_{\infty}^{-}\varphi\right)  (x_{0}).
\]

Hence, noticing that%
\[
A_{n}^{p_{n}-1}\leq A_{n}^{p_{n}-1}+C_{n}^{p_{n}-1}\leq B_{n}^{p_{n}-1}%
\]
we conclude that%
\[
\left(  \mathcal{L}_{\infty}\varphi\right)  (x_{0})=\left(  \mathcal{L}%
_{\infty}^{+}\varphi\right)  (x_{0})+\left(  \mathcal{L}_{\infty}^{-}%
\varphi\right)  (x_{0})\leq0
\]
since%
\[
\left(  \mathcal{L}_{\infty}^{+}\varphi\right)  (x_{0})=\lim_{n\rightarrow
\infty}A_{n}\leq\lim_{n\rightarrow\infty}B_{n}=-\left(  \mathcal{L}_{\infty
}^{-}\varphi\right)  (x_{0}).
\]
We have proved that $u_{\infty}$ is a supersolution of (\ref{viscinf}).
Therefore, by directly applying Lemma 22 of \cite{LindLind} we conclude
$u_{\infty}>0$ in $\Omega.$

The strict positiveness of $u_{\infty}$ in $\Omega$ and the uniqueness of the
nonnegative minimizers of $\left\vert \cdot\right\vert _{s}$ on $\mathcal{M}%
_{s}$ imply that if $w\in\mathcal{M}_{s}$ is such that
\[
\left\vert w\right\vert _{s}=\min_{u\in\mathcal{M}_{s}}\left\vert u\right\vert
_{s}%
\]
then $\left\vert w\right\vert =v_{s}=(k_{\infty})^{-1}u_{\infty}>0$ in
$\Omega$ (recall that $\left\vert w\right\vert $ is also a minimizer). The
continuity of $w$ then implies that either $w>0$ in $\Omega$ or $w<0$ in
$\Omega.$ Consequently, $w=v_{s}$ or $w=-v_{s}.$

Now, recalling that%
\[
\lim_{n\rightarrow\infty}(\Lambda_{p_{n}})^{\frac{1}{p_{n}-1}}=\left\vert
u_{\infty}\right\vert _{s}%
\]
and using that $\omega(x_{0})>0$ and $u_{\infty}(x_{0})>0$ we have
\[
\lim_{n\rightarrow\infty}C_{n}=\left\vert u_{\infty}\right\vert _{s}%
\]

Hence, since
\[
C_{n}^{p_{n}-1}\leq A_{n}^{p_{n}-1}+C_{n}^{p_{n}-1}\leq B_{n}^{p_{n}-1},
\]
we obtain
\[
\left\vert u_{\infty}\right\vert _{s}=\lim_{n\rightarrow\infty}C_{n}\leq
\lim_{n\rightarrow\infty}B_{n}=-\left(  \mathcal{L}_{\infty}^{-}%
\varphi\right)  (x_{0}).
\]
It follows that $u_{\infty}$ is a viscosity supersolution of (\ref{visc}).

Now, let us take a pair $\left(  x_{0},\varphi\right)  \in\Omega\times
C_{0}^{1}(\mathbb{R}^{N})$ satisfying
\begin{equation}
\varphi(x_{0})=u_{\infty}(x_{0})\quad\mathrm{and}\quad\varphi(x)\geq
u_{\infty}(x)\quad\forall\,x\in\mathbb{R}^{N}. \label{visc2}%
\end{equation}

Since
\[
-\left\vert u_{\infty}\right\vert _{s}\leq\frac{u_{\infty}(x)-u_{\infty}%
(x_{0})}{\left\vert x-x_{0}\right\vert ^{s}}\leq\frac{\varphi(x)-\varphi
(x_{0})}{\left\vert x-x_{0}\right\vert ^{s}}\quad\forall\,x\in\mathbb{R}%
^{N}\setminus\left\{  x_{0}\right\}  ,
\]
we have%
\[
-\left\vert u_{\infty}\right\vert _{s}\leq\inf_{x\in\mathbb{R}^{N}%
\setminus\left\{  x_{0}\right\}  }\frac{\varphi(x)-\varphi(x_{0})}{\left\vert
x-x_{0}\right\vert ^{s}}=\left(  \mathcal{L}_{\infty}^{-}\varphi\right)
(x_{0}).
\]
Therefore, $u_{\infty}$ is a viscosity subsolution of (\ref{visc}).
\end{proof}

Since $v_{s}=(k_{\infty})^{-1}u_{\infty}$ is the only positive minimizer of
$\left\vert \cdot\right\vert _{s}$ on $C_{0}^{0,s}(\overline{\Omega}%
)\setminus\left\{  0\right\}  $ and $\mathcal{L}_{\infty}^{-}(ku)=k\mathcal{L}%
_{\infty}^{-}u\ $\ for any positive constant $k,$ the following corollary is immediate.

\begin{corollary}
\label{viscv}The minimizer $v_{s}$ is a viscosity solution of the problem%
\[
\left\{
\begin{array}
[c]{lll}%
\mathcal{L}_{\infty}^{-}u+\mu_{s}=0 & \mathrm{in} & \Omega\\
u=0 & \mathrm{in} & \mathbb{R}^{N}\setminus\Omega.
\end{array}
\right.
\]

\end{corollary}

\section{Acknowledgements}

G. Ercole was partially supported by CNPq/Brazil (306815/2017-6) and
Fapemig/Brazil (CEX-PPM-00137-18). R. Sanchis was partially supported by
CNPq/Brazil (310392/2017-9) and Fapemig/Brazil (CEX-PPM-00600-16). G. A.
Pereira was partially supported by Capes/Brazil (Finance Code 001).

\end{document}